\documentclass[12pt,reqno]{amsart}
\usepackage{amsthm}
\usepackage{amscd}
\usepackage{amsfonts}
\usepackage{amssymb}
\usepackage{amsgen}
\usepackage{amsmath}
\usepackage{amsopn}

\theoremstyle{plain}
\newtheorem{thm}{Theorem}[section]
\newtheorem{lem}[thm]{Lemma}
\newtheorem{prop}[thm]{Proposition}
\newtheorem{cor}[thm]{Corollary}
\newtheorem{conj}[thm]{Conjecture}
\theoremstyle{definition}

 \newcommand{\nc}{\newcommand}

 \nc{\Z}{{\mathbb Z}}
 \nc{\R}{{\mathbb R}}
 \nc{\N}{{\mathbb N}}
 \nc{\ZN}{{{\mathbb N}_0}}
 \nc{\Q}{{\mathbb Q}}
 \nc{\CC}{{\mathbb C}}

 \nc{\calP}{{\mathcal P}}
 \nc{\gam}{{\gamma}}
 \nc{\gG}{{\Gamma}}
 \nc{\om}{{\omega}}
 \nc{\vep}{{\varepsilon}}
 \nc{\ga}{{\alpha}}
 \nc{\gl}{{\lambda}}
 \nc{\gb}{{\beta}}
 \nc{\gd}{{\delta}}
 \nc{\bfs}{{\bf s}}
 \nc{\gs}{{\sigma}}
 \nc{\gth}{{\theta}}
 \nc{\gS}{{\Sigma}}
 \nc{\gk}{{\kappa}}
  \nc{\gz}{{\zeta}}
 \nc{\tgz}{{\tilde{\zeta}}}
 \nc{\gO}{{\Omega}}
 \nc{\sif}{{\mathcal S}}
 \nc{\gt}{{\tau}}
 \nc{\Lra}{\Longrightarrow}
 \nc{\lra}{\longrightarrow}
 \nc{\lmaps}{\longmapsto}
 \nc{\fS}{{\mathfrak S}}
 \nc{\DD}{{\mathfrak D}}
 \nc{\Llra}{\Longleftrightarrow}
 \nc{\ol}{\overline}
 \nc{\ola}{\overleftarrow}
 \nc{\lms}{\longmapsto}
 \nc{\cv}{{{\mathsf c}{\mathsf v}}}
 \nc{\zq}{{\zeta_q}}
 \nc\qup{{q\uparrow 1}}
 \nc{\us}{\underset}
 \nc{\tn}{{\tilde{n}}}
 \nc{\gD}{{\Delta}}
 \nc{\bi}{{\bf i}}
 \nc{\bfone}{{\bf 1}}
\begin{document}
\title{A Family of Supercongruences Involving Multiple Harmonic Sums}

\author{Megan McCoy$^*$,
Kevin Thielen$^*$,
Liuquan Wang$^{\dagger}$,
and Jianqiang Zhao$^{\star}$}
\date{}

\email{wanglq@whu.edu.cn}
\email{zhaoj@ihes.fr}

\address{$^*$Department of Mathematics, Eckerd College, St.\ Petersburg, FL 33711, USA}
\address{$^{\dagger}$Department of Mathematics, National University of Singapore, Singapore, 119076, Singapore}
\address{$^\star$Department of Mathematics, The Bishop's School, San Diego, CA 92037}



\maketitle
\allowdisplaybreaks

\begin{abstract}
In recent years, the congruence
$$
\sum_{\substack{i+j+k=p\\ i,j,k>0}} \frac1{ijk} \equiv -2 B_{p-3} \pmod{p},
$$
first discovered by the last author have been generalized by either
increasing the number of indices and considering the corresponding
supercongruences, or by considering the alternating version of multiple harmonic sums.
In this paper, we prove a family of similar supercongruences modulo prime powers $p^r$
with the indexes summing up to $mp^r$ where $m$ is coprime to $p$, where
all the indexes are also coprime to $p$.
\end{abstract}

\section{Introduction}
Multiple harmonic sums are multiple variable generalization of harmonic numbers. Let $\N$
be the set of natural numbers. For $\bfs=(s_1,\dots, s_d)\in \N^d$ and any $N\in\N$,
we define the multiple harmonic sums (MHS) by
\begin{equation*}
 H_N(\bfs):=\sum_{N\ge k_1>\dots>k_d>0} \prod_{i=1}^d \frac{1}{k_i^{s_i}}.
\end{equation*}
Since mid 1980s these sums have appeared in a few diverse areas of mathematics
as well as theoretical physics such as multiple zeta values \cite{HessamiPilehrood2Ta2013,Hoffman2005,KanekoZa2013},
Feynman integrals \cite{Blumlein1999,DevotoDu1984}, quantum electrodynamics
and quantum chromodynamics \cite{Blumlein2005,Vermaseren1999}.

In \cite{Zhao2008a} the last author started to investigate congruence properties of MHSs,
which were also considered by Hoffman \cite{Hoffman2005} independently.
As a byproduct, the following
intriguing congruence was noticed: for all primes $p\ge 3$
\begin{equation}\label{equ:BaseCongruence}
  \sum_{\substack{i+j+k=p\\ i,j,k>0}} \frac1{ijk} \equiv -2 B_{p-3} \pmod{p},
\end{equation}
where $B_k$ are Bernoulli numbers defined by the generating series
$$\frac{t}{e^t-1}=\sum_{k=0}^\infty B_k\frac{t^k}{k!}.$$
This was proved by the last author using MHSs in \cite{Zhao2007b}, and by Ji using
some combinatorial identities in \cite{Ji}. Later on, a few generalizations and analogs
were obtained by either increasing the number of indices and considering the corresponding
supercongruences (see \cite{WangCa2014,Wang2014b,XiaCa2010,ZhouCa2007}), or considering the alternating
version of MHSs (see \cite{ShenCai2012b,Wang2014a}).

Let $\mathcal{P}_{n}$ be the set of positive integers not divisible by $n$.
To generalize the congruence in \eqref{equ:BaseCongruence}, we wonder if
for every \emph{odd} integer $d\ge 3$
there exists a rational number $q_d$ such that
\begin{equation}\label{equ:CongruenceDepthd}
\sum_{\substack{l_1+l_2+\dots+l_d=p^r \\ l_1,\dots,l_d\in \calP_p}}
    \frac{1}{l_1 l_2 \dots l_d} \equiv q_d \cdot p^{r-1} B_{p-d} \pmod{p^r}
\end{equation}
for any prime $p>d$ and integer $r\ge 2$.
In \cite{WangCa2014,Wang2014b} it is shown that $q_3=-2$ and $q_5=-5!/6$.
We should point it out that when $d$ is even, the congruence pattern
is quite different, see \cite{Wang2015,Zhao2014}.
In this paper, we shall prove the following main result when $d=7$.
\begin{thm}\label{thm:main}
Let $r$ and $m$ be positive integers and $p>7$ be a prime such that $p\nmid m$.
\begin{enumerate}
  \item[\upshape{(i)}] If $r=1$, then
\begin{equation*}
\sum_{\substack{l_1+l_2+\dots+l_7=mp \\ l_1,\dots,l_7\in \calP_p}}
    \frac{1}{l_1 l_2\dots l_7} \equiv -(504m+210m^3+6m^5) B_{p-7} \pmod{p}.
\end{equation*}

  \item[\upshape{(ii)}] If $r \ge 2$, then
\begin{equation*}
\sum_{\substack{l_1+l_2+\dots+l_7=mp^r\\  l_1,\dots,l_7\in \calP_p}}
    \frac{1}{l_1 l_2\dots l_7} \equiv -\frac{7!}{10}\cdot m p^{r-1} B_{p-7} \pmod{p^r}.
\end{equation*}
\end{enumerate}
\end{thm}
To establish this result, for all positive integers $n$, $m$, $r$ and primes $p$,
following the notation in \cite{Wang2014b}, we define
\begin{equation*}
S_n^{(m)} (p^r):=\sum_{\substack{l_1+l_2+\dots+l_n=mp^r\\ p^r>l_1,\dots,l_n\in \calP_p}}
    \frac{1}{l_1 l_2\dots l_n}.
\end{equation*}
Notice that the sum in the theorem is not exactly the same type as that appearing in $S_n^{(m)}$
since the condition $p^r>l_i$ for all $i$ is not present.
The main idea of our proof is to show the special case when $m=1$ first. In order
to do this we will first prove the relation
\begin{equation}\label{equ:kInductionStep}
S_n^{(1)}(p^{r+1}) \equiv pS_n^{(1)} (p^{r}) \quad \pmod{p^{r+1}}, \quad \forall r\ge 2,
\end{equation}
and then use induction.
Notice that when $r=1$ the above congruence usually does not hold anymore. So
we will compute the congruence of $S_n^{(1)}(p^2)$ and $S_n^{(1)}(p)$ separately
by relating them to the following quantities:
\begin{equation*}
R_n^{(m)}(p):=\sum_{\substack{l_1+l_2+\dots+l_n=mp\\ l_1,\dots,l_n\in \calP_p }}
    \frac{1}{l_1 l_2\dots l_n}.
\end{equation*}

To save space, throughout the paper when the prime $p$ is fixed we often use the shorthand
$H(\bfs)= H_{p-1}(\bfs)$. Moreover, we shall also need the modified sum
\begin{equation*}
 H^{(p)}_N(\bfs):=\sum_{\substack{N\ge k_1>\dots>k_d>0\\ k_1,\dots,k_d\in \calP_p }} \ \prod_{i=1}^d \frac{1}{k_i^{s_i}}.
\end{equation*}

\section{Preliminary lemmas}
Let $C^{(m)}_{a,p}(n)$ denote the number of solutions $(x_1,\dots,x_n)$ of the equation
$$x_1 + \dots + x_n =mp-a, \quad 0 \le x_i < p\ \forall i=1,\dots,n.$$
For all $b\ge 1$ set
$$
  \gb_n(a,b):=\binom{bp-a+n-1}{n-1}\qquad\text{and} \qquad
  \gam_n(a):=\frac{(-1)^{a-1}}{a \binom{n-1}{a} }.
$$
It is not hard to see that
\begin{equation}\label{equ:betaCongruentGamma}
\gb_n(a,b) \equiv \frac{b(-1)^{a-1} (n-a-1)!(a-1)! }{(n-1)!} p \
\equiv \frac{b(-1)^{a-1}}{a \binom{n-1}{a} }p
  \equiv b\gam_n(a) p
  \pmod{p^2}.
\end{equation}

\begin{lem}\label{lem:Cn(m)a}
For all $m,n,a\in\N$ and primes $p$, we have
\begin{equation*}
C^{(m)}_{a,p}(n) \equiv (-1)^{m-1}\binom{n-2}{m-1} \gam_n(a) p
\equiv (-1)^{m-1}\binom{n-2}{m-1} C^{(1)}_{a,p}(n) \pmod{p^2}.
\end{equation*}
\end{lem}

\begin{proof}
The coefficient of $x^{mp-a}$ in the expansion of
$\big(1+x+\cdots +x^{p-1}\big)^n=(x^p-1)^n (x-1)^{-n}$ is
\begin{displaymath}
\begin{split}
C^{(m)}_{a,p}(n)& =\sum_{i=0}^{m} \binom{n}{i}\binom{-n}{mp-ip-a} (-1)^{mp-a} \\
& =\sum_{i=0}^{m}\binom{n}{i}(-1)^{ip}\binom{n+mp-ip-a-1}{n-1} \\
& =\sum_{i=0}^{m}(-1)^{i}\binom{n}{i}\binom{n+mp-ip-a-1}{n-1}\\
& \equiv \sum_{i=0}^{m}(-1)^{i}\binom{n}{i}(m-i)  \gam_n(n-a) p\pmod{p^2}
\end{split}
\end{displaymath}
by \eqref{equ:betaCongruentGamma}. Now we calculate the sum
\begin{equation*}
A(m)=\sum_{i=0}^{m}{(-1)^{i}\binom{n}{i}(m-i)}.
\end{equation*}
It is easy to see that $A(m)$ is the coefficient of $x^m$ in the expansion of
\begin{equation*}
(1-x)^{n}\cdot\sum_{i=0}^{\infty}{ix^{i}}
=(1-x)^{n}\cdot\frac{x}{(1-x)^{2}}
=x(1-x)^{n-2}
=\sum_{m=1}^{n-1} (-1)^m\binom{n-2}{m-1} x^m,
\end{equation*}
as desired.
\end{proof}

\begin{cor}
When $n=7$, we have
\begin{align*}
C^{(2)}_{1,p}(7) - C^{(2)}_{6,p}(7) &\equiv -(5/3)p, \quad
C^{(3)}_{1,p}(7) - C^{(3)}_{6,p}(7) \equiv \phantom{,}(10/3)p \pmod{p^2}, \\
C^{(3)}_{2,p}(7) - C^{(3)}_{5,p}(7) &\equiv -(2/3)p, \quad
C^{(2)}_{2,p}(7) - C^{(2)}_{5,p}(7) \equiv \phantom{-} (1/3)p \pmod{p^2}, \\
C^{(3)}_{3,p}(7) - C^{(3)}_{4,p}(7) &\equiv \phantom{-}(1/3)p, \quad
C^{(2)}_{3,p}(7) - C^{(2)}_{4,p}(7) \equiv -(1/6)p \pmod{p^2}.
\end{align*}
\end{cor}

Part (ii) of the following lemma generalizes \cite[Lemma 1(ii)]{Wang2014b}.
\begin{lem} \label{lem:recurrence}
Let $1\le k\le n-1$ and $p>n$ a prime. For all $r\ge 1$, we have
\begin{enumerate}
  \item[\upshape (i)] $S_n^{(k)}(p^{r})\equiv (-1)^n S_n^{(n-k)}(p^{r})$ \text{\rm{(mod $p^r$)}}.
  \item[\upshape (ii)] $\displaystyle S_n^{(m)}(p^{r+1})\equiv \sum_{a=1}^{n-1}
    C^{(m)}_{a,p}(n) S_n^{(a)}(p^{r}) \pmod{p^{r+1}}.$
\end{enumerate}
\end{lem}
\begin{proof} (i) can be found in \cite{Wang2014b}. We now prove (ii).
For any $n$-tuples $(l_1,\cdots ,l_n)$ of integers satisfying $l_1+\cdots +l_n=mp^{r+1}$,
$p^{r+1}> l_i \in \calP_p$, $1 \le i \le n$, we rewrite them as
\begin{equation*}
l_i=x_ip^r+y_i, \quad 0\le x_i<p, \quad 1\le y_i<p^r, \quad y_i\in {\calP_{p}}, \quad 1 \le i \le n.\end{equation*}
Since
\begin{equation*}
\Big(\sum_{i=1}^n{x_i}\Big)p^r+\sum_{i=1}^n{y_i}=mp^{r+1}
\end{equation*}
and $n<p$, we know there exists $1\le a<n$ such that
\begin{equation*}\left\{ \begin{array}{ll}
 x_1+\cdots +x_n=mp-a, \quad 0\le x_i<p,\\
 y_1+\cdots +y_n=ap^r. \\
\end{array} \right. \end{equation*}
For $1 \le a <n$, the equation $x_1+ \cdots + x_n=mp-a$ has $C^{(m)}_{a,p}(n)$
integer solutions with $0\le x_i<p$. Hence by Lemma~\ref{lem:Cn(m)a}
\begin{equation*}
\begin{split}
S_n^{(m)}(p^{r+1})  & =\sum_{\substack
 {l_1+\cdots +l_n=mp^{r+1} \\
 l_1,\cdots ,l_n\in\calP_p}}  \frac{1}{l_1l_2\cdots l_n}  \\
 & =\sum_{a=1}^{n-1}\sum_{\substack
 {x_1+\cdots +x_n=mp-a \\
 0\le x_i<p}} \ \sum_{\substack
 { y_1+\cdots +y_n=ap^r \\
 y_i\in\calP_p,\, y_i<p^r}} \frac{1}{ (x_1p^r+y_1)\cdots (x_np^r+y_n) }   \\
 & \equiv \sum_{a=1}^{n-1}\sum_{\substack
 {x_1+\cdots +x_n=mp-a \\
 0\le x_i<p}}\ \sum_{\substack
 { y_1+\cdots +y_n=ap^r \\
 y_i\in\calP_p,\, y_i<p^r}} \left(1-\frac{x_1}{y_1}p^r-\cdots-\frac{x_n}{y_n}p^r\right)
 \frac1{ y_1\cdots y_n}   \pmod{p^{r+1}}\\
 & \equiv \sum_{a=1}^{n-1}
    C^{(m)}_{a,p}(n) S_n^{(a)}(p^{r}) \pmod{p^{r+1}}
\end{split}
\end{equation*}
since for each $x_j$ ($j=1,\dots,n$), we have
\begin{equation*}
\sum_{\substack{x_1+\dots+x_n=mp-a \\ 0 \le x_i<p}} x_j
= \frac{1}{n} \sum_{x_1+\dots+x_n=mp-a} (x_1+x_2+\dots+x_n)
= \frac{mp-a}{n} C^{(m)}_{a,p}(n) \equiv 0  \pmod{p}
\end{equation*}
by Lemma~\ref{lem:Cn(m)a}.
\end{proof}

\section{Congruences involving multiple harmonic sums}
We first consider some un-ordered sums. Lemmas \ref{lem:Ub} and \ref{lem:homo1} were proved by Zhou and Cai~\cite{ZhouCa2007}.
\begin{lem}\label{lem:Ub}
Let $p$ be a prime and $\ga_1,\dots,\ga_n$ be positive integers, $r=\ga_1+\dots+\ga_n\le p-3$. Define the un-ordered sum
\begin{equation*}
U_b(\ga_1,\dots,\ga_n)=\sum_{\substack{0<l_1,\dots,l_n<bp \\ l_i\ne l_j \forall i\ne j, l_i\in\calP_p}}
\frac{1}{l_1^{\ga_1}\cdots l_n^{\ga_n}}.
\end{equation*}
Then
\begin{equation*}
U_1(\ga_1,\dots,\ga_n) \equiv
\left\{
  \begin{array}{ll}
    \displaystyle (-1)^n (n-1)! \frac{r(r+1)}{2(r+2)} B_{p-r-2}\cdot p^2  &\pmod{p^3}, \quad \hbox{if $r$ is odd;} \\
     \displaystyle (-1)^{n-1} (n-1)! \frac{r}{r+1} B_{p-r-1}\cdot p &\pmod{p^2},  \quad \hbox{if $r$ is even.}
  \end{array}
\right.
\end{equation*}
\end{lem}
This easily leads to the following corollary (see also \cite{Zhao2008a}).
\begin{cor}\label{cor:homo}
Let $p$ be a prime and $\ga$ be positive integer. Then
\begin{equation*}
H(\{\ga\}^n) \equiv
\left\{
  \begin{array}{ll}
    \displaystyle (-1)^n\frac{\ga(n\ga+1)}{2(n\ga+2)} B_{p-n\ga-2}\cdot p^2
                        &\pmod{p^3}, \quad \hbox{if $n\ga$ is odd;} \\
     \displaystyle (-1)^{n-1}\frac{\ga}{n\ga+1} B_{p-n\ga-1}\cdot p
                        &\pmod{p^2},  \quad \hbox{if $n\ga$ is even.}
  \end{array}
\right.
\end{equation*}
\end{cor}

\begin{lem}\label{lem:homo1}
Let $n>1$ be positive integer and let $p>n+1$ be a prime. Then
\begin{equation*}
R_n^{(1)}(p)=\sum_{\substack{l_1+\dots+l_n=p \\ l_1,\dots, l_n>0}}
\frac{1}{l_1 \cdots l_n} \equiv
\left\{
  \begin{array}{ll}
    \displaystyle -(n-1)! B_{p-n}   &  \pmod{p},\quad \hbox{if $n$ is odd;} \\
     \displaystyle-\frac{n\cdot n!}{n+1} B_{p-n-1} p &\pmod{p^2}, \quad  \hbox{if $n$ is even.}
  \end{array}
\right.
\end{equation*}
\end{lem}

The next result generalizes Lemma \ref{lem:Ub}.
\begin{lem}
Let $p$ be a prime and $\ga_1,\dots,\ga_n$ be positive integers, $r=\ga_1+\dots+\ga_n\le p-3$. Then
\begin{equation*}
U_b(\ga_1,\dots,\ga_n)  \equiv
\left\{
  \begin{array}{ll}
    \displaystyle (-1)^n (n-1)! \frac{b^2 r(r+1)}{2(r+2)} B_{p-r-2}\cdot p^2 & \pmod{p^3},\quad  \hbox{if $r$ is odd;} \\
     \displaystyle (-1)^{n-1} (n-1)! \frac{br}{r+1} B_{p-r-1} \cdot p  &\pmod{p^2},\quad  \hbox{if $r$ is even.}
  \end{array}
\right.
\end{equation*}
\end{lem}
\begin{proof} For all $k\ge 1$, we have
\begin{align*}
\sum_{kp<l<(k+1)p} \frac{1}{l^\ga}
=&\ \sum_{l=1}^{p-1} \frac{1}{(l+kp)^\ga}
=\sum_{l=1}^{p-1} \frac{1}{(1+kp/l)^\ga}\frac{1}{l^\ga} \\
\equiv &\  \sum_{l=1}^{p-1} \left(1-\frac{\ga kp}l +\frac{\ga(\ga+1)}{2l^2} k^2p^2\right)\frac{1}{l^\ga} \pmod{p^3}\\
\equiv &\ \sum_{l=1}^{p-1}\frac{1}{l^\ga}   -\ga kp \sum_{l=1}^{p-1} \frac{1}{l^{\ga+1}} \pmod{p^3}.
\end{align*}
By Lemma~\ref{lem:Ub} we see that
\begin{equation*}
    \sum_{kp<l<(k+1)p} \frac{1}{l^\ga}
\equiv \left\{
  \begin{array}{ll}
    \displaystyle - \frac{\ga(\ga+1)}{\ga+2}\left(\frac12+k\right) B_{p-\ga-2} p^2 & \pmod{p^3}, \quad \hbox{if $\ga$ is odd;} \\
     \displaystyle \frac{\ga}{\ga+1} B_{p-\ga-1} p  &\pmod{p^2}, \quad \hbox{if $\ga$ is even.}
  \end{array}
\right.
\end{equation*}
Therefore for any positive integer $b$, we have
\begin{equation*}
 \sum_{0<l<b p,\, p\nmid l} \frac{1}{l^\ga}
\equiv \left\{
  \begin{array}{ll}
    \displaystyle - \frac{b^2 \ga(\ga+1)}{2(\ga+2)}  B_{p-\ga-2} p^2 & \pmod{p^3}, \quad \hbox{if $\ga$ is odd;} \\
     \displaystyle \frac{b\ga}{\ga+1} B_{p-\ga-1} p  &\pmod{p^2}, \quad \hbox{if $\ga$ is even.}
  \end{array}
\right.
\end{equation*}
This proves the lemma in the case $n=1$. Now assume the lemma holds when the number of variables
is less than $n$. Then
\begin{multline*}
U_b\big(\ga_1,\dots,\ga_n) =\sum_{\substack{1 \leq l_1,\dots,l_{n-1} <bp \\l_i \neq l_j,\, l_i\in\calP_p}} \frac{1}{l_1^{\ga_1}\dots l_{n-1}^{\ga_{n-1}}} \left(\sum_{1\leq l_n<bp,\, l_{n}\in\calP_p} \frac{1}{l_n^{\ga_n}} - \sum_{i=1}^{n-1} \frac{1}{l_i^{\ga_n}}\right)\\
\equiv  U_b\big(\ga_1,\dots,\ga_{n-1})\left( \sum_{1\leq l_n<bp,\, l_{n}\in\calP_p} \frac{1}{l_n^{\ga_n}}\right)
-  \sum_{i=1}^{n-1} U_b\big(\ga_1,\dots,\ga_{i-1},\ga_i+\ga_n,\ga_{i+1},\dots,\ga_{n-1}\big).
\end{multline*}
By the induction assumption, we have
\begin{equation*}
 U_b\big(\ga_1,\dots,\ga_{n-1}) \sum_{1 \leq l_n <bp,\, l_n\in\calP_p} \frac{1}{l_n^{\ga_n}} \equiv
 \left\{
  \begin{array}{ll}
   0 & \pmod{p^3}, \quad \hbox{if $r$ is odd;} \\
   0 & \pmod{p^2}, \quad \hbox{if $r$ is even.}
  \end{array}
 \right.
\end{equation*}
Thus if $r$ is odd, we have
\begin{align*}
 U_b\big(\ga_1,\dots,\ga_n) &\ \equiv -(n-1)  U_b\big(\gb_1,\dots,\gb_{n-1})
                \qquad \Big(\text{here } \sum_{j=1}^{n-1} \gb_j=r\Big)\\
& \ \equiv -(n-1) (-1)^{n-1} (n-2)! \frac{b^2 r(r+1)}{2(r+2)} p^2 B_{p-r-2} \pmod{p^3}\\
& \ \equiv (-1)^n (n-1)! \frac{b^2 r(r+1)}{2(r+2)} p^2 B_{p-r-2} \pmod{p^3}.
\end{align*}
Similarly, if $r$ is even, we can derive
\begin{equation*}\nonumber
 U_b\big(\ga_1,\dots,\ga_n) \equiv (-1)^{n-1}(n-1)! \frac{br}{r +1} p B_{p-r-1} \pmod{p^2}.
\end{equation*}
\end{proof}

\begin{lem}\label{lem:2pGeneral}
Let $n$ be an odd positive integer and $p$ be a prime. Then
\begin{equation*}
R_n^{(2)}(p)=\sum_{\substack{l_1+\dots+l_n=2p\\ l_1,\dots,l_n\in \calP_p }} \frac{1}{l_1\dots l_n}
\equiv -\frac{n+1}{2}  \cdot (n-1)!B_{p-n}  \pmod{p}.
\end{equation*}
\end{lem}

\begin{proof} We have
\begin{alignat*}{4}
&\ \sum_{\substack{l_1+\dots+l_n=2p\\ l_1,\dots,l_n\in \calP_p }} \frac{1}{l_1\dots l_n} \\
=&\ \sum_{l_1+\dots+l_n=2p} \frac{1}{l_1\dots l_n}
-\frac{n}{p}\sum_{l_1+\dots+l_{n-1}=p}\frac{1}{l_1\dots l_{n-1}} \\
=&\ \frac{n!}{2p}\sum_{0<u_1<\dots<u_{n-1}<2p} \frac{1}{u_1\dots u_{n-1}}
-\frac{n!}{p^2}\sum_{0<u_1<\dots<u_{n-2}<p}\frac{1}{u_1\dots u_{n-2}} \\
=&\ \frac{n!}{2p}H^{(p)}_{2p-1}(\{1\}^{n-1})
+ \frac{n!}{2p^2}\sum_{j=1}^{n-1}
\sum_{\substack{0<u_1<\dots<u_{j-1}<p\\
        p<u_{j+1}<\dots<u_{n-1}<2p}}\frac{1}{u_1\dots u_{j-1}u_{j+1}\dots u_{n-1}}
- & \frac{n!}{p^2}H(\{1\}^{n-2}) &   \\
\equiv&\ \frac{n!}{2p}\cdot \frac{U_{2}(\{1\}^{n-1})}{(n-1)!}
+\frac{n!}{2p^2}\left(2H(\{1\}^{n-2})-p\frac{U_1(2,\{1\}^{n-3})}{(n-3)!}\right) \\
+&\  \frac{n!}{2p^2}\sum_{j=2}^{n-2} H(\{1\}^{j-1})
\left(H(\{1\}^{n-j-1})-p\frac{U_1(2,\{1\}^{n-j-2})}{(n-j-2)!}\right) -\frac{n!}{p^2}H(\{1\}^{n-2}) &\pmod{p} &\  \\
\equiv&\ \frac{n}{2p}U_{2}(\{1\}^{n-1}) -\frac{n!}{2p}\frac{U_1(2,\{1\}^{n-3})}{(n-3)!} &\pmod{p}&\ \\
\equiv&\  -\frac{n}{p}(n-2)!\frac{n-1}{n} B_{p-n}p
   -(-1)^{n-3}\frac{n!}{2p}(n-3)! \frac{n-1}{n (n-3)!} B_{p-n}p  &\pmod{p} &\  \\
\equiv&\  -\frac{n+1}{2} (n-1)! B_{p-n}  &\pmod{p}  &, \
\end{alignat*}
as desired.
\end{proof}

\begin{cor} \label{cor:S_n(2)(p)forOddn}
Let $n$ be an odd positive integer with $n\ge 5$. Then for all prime $p>n$, we have
\begin{equation*}
S_n^{(2)}(p)\equiv \frac{n-1}{2} \cdot (n-1)!B_{p-n} \pmod{p}.
\end{equation*}
\end{cor}

\begin{proof}
We observe that
\begin{equation*}
\sum_{\substack{l_1+\dots+l_n=2p\\ l_j\in \calP_p \, \forall j }} \frac{1}{l_1\dots l_n}\equiv \sum_{\substack{l_1+\dots+l_n=2p\\ l_1,\dots,l_n <p }} \frac{1}{l_1\dots l_n} +n \sum_{\substack{l_1+\dots+l_n=p\\ l_1,\dots,l_n <p }} \frac{1}{(l_1+p) l_2 \dots l_n} \pmod{p}.
\end{equation*}
By Lemma~\ref{lem:homo1}, we have $S_n^{(1)}(p)\equiv -(n-1)! B_{p-n}\pmod{p}$. So we deduce
\begin{equation*}
S_n^{(2)}(p)\equiv \sum_{\substack{l_1+\dots+l_n=2p\\  l_j \in \calP_p \, \forall j}}
    \frac{1}{l_1 l_2 \dots l_n}- nS_n^{(1)}(p)\equiv \frac{n-1}{2} \cdot (n-1)!B_{p-n} \pmod{p}.
\end{equation*}
\end{proof}

\begin{lem}\label{lem:3pGeneral}
Let $n\ge 3$ be an odd positive integer.
Then for all prime $p\ge \max\{n,5\}$, we have
\begin{eqnarray*}
&&R_n^{(3)}(p)=\sum_{\substack{l_1+\dots+l_n=3p\\ l_1,\dots,l_n\in \calP_p }} \frac{1}{l_1\dots l_n} \\
&\equiv & -\frac{1}{n}{n+2 \choose 3}\cdot (n-1)!B_{p-n}-\frac{n!}{6}\sum\limits_{\substack{a+b+c=\frac{n-3}{2} \\a,b,c \ge 1}} \frac{B_{p-2a-1}B_{p-2b-1}B_{p-2c-1}}{(2a+1)(2b+1)(2c+1)}  \pmod{p}.
\end{eqnarray*}
\end{lem}

\begin{proof}
Let $\nu=n-1$ throughout the proof. Let $u_i=l_1+\dots+l_i$, $1 \le i \le \nu$. We have
\begin{equation}\label{starteq}
\sum_{\substack{l_1+\dots+l_n=3p\\ l_1,\dots,l_n \in \calP_p }} \frac{1}{l_1\dots l_n} =\frac{n!}{3p}
\sum_{\substack{1 \le u_1< \dots <u_{n-1}<3p \\ u_1,u_2-u_1,\dots, u_{\nu}-u_{n-2},u_\nu \in \calP_p }} \frac{1}{u_1\dots u_\nu}.
\end{equation}
Evidently
\begin{align} \notag
&\ \ \ \sum_{\substack{1 \le u_1< \dots <u_\nu<3p \\ u_1,u_2-u_1,\dots,u_\nu-u_{n-2},u_\nu \in \calP_p}}
\frac{1}{u_1\dots u_\nu} \\
&\ =  \sum_{\substack{1\le u_1<\dots<u_\nu<3p \\ u_1,u_2,\dots,u_\nu \in \calP_p \\ u_2-u_1,\dots,u_\nu-u_{n-2} \in \calP_p}} \frac{1}{u_1\dots u_\nu}
+\sum_{i=2}^{n-4}\sum_{j=i+2}^{n-2} \sum_{\substack{1\le u_1<\dots<u_\nu<3p, \: u_i=p,u_j=2p
 \\ \forall k\ne i,k\ne j, u_k,u_2-u_1,\dots, u_{\nu}-u_{n-2} \in \calP_p }} \frac{1}{u_1\dots u_\nu}
                                                                                \notag  \\
&\ +\sum_{j=2}^{n-2} \sum_{\substack{1\le u_1<\dots<u_\nu<3p, \: u_j=p
        \\ \forall k\ne j, u_k,u_2-u_1,\dots, u_{\nu}-u_{n-2} \in \calP_p }} \frac{1}{u_1\dots u_\nu}
+\sum_{j=2}^{n-2} \sum_{\substack{1\le u_1<\dots<u_\nu<3p, \: u_j=2p
        \\ \forall k\ne j, u_k,u_2-u_1,\dots, u_{\nu}-u_{n-2} \in \calP_p }} \frac{1}{u_1\dots u_\nu}.
        \label{equ:3pcaseSomeArePMultiple}
\end{align}
Now we deal with the sums in (\ref{equ:3pcaseSomeArePMultiple}) one by one.  For $2\le j\le n-2$,
\begin{align}\label{sum1}
\ &\sum_{\substack{1\le u_1<\dots<u_\nu<3p, \: u_j=p
        \\ \forall k\ne j, u_k,u_2-u_1,\dots, u_{\nu}-u_{n-2} \in \calP_p }} \frac{1}{u_1\dots u_\nu} \nonumber \\
=&\ \frac{1}{p} \sum_{1 \le u_1<\cdots<u_{j-1}<p} \frac{1}{u_1\cdots u_{j-1}}
\sum_{\substack{p<u_{j+1}<\cdots<u_\nu<3p,
        \\ \forall k>j, u_k,u_{j+2}-u_{j+1},\dots, u_{\nu}-u_{n-2} \in \calP_p }}
         \frac{1}{u_{j+1}\cdots u_\nu} \nonumber \\
=&\ \frac{1}{p} H(\{1\}^{j-1})
\sum_{\substack{0<u_1<\cdots<u_{\nu-j}<2p\\ \forall k, u_k,u_2-u_1,\dots, u_{\nu-j}-u_{\nu-j-1} \in \calP_p }}
             \frac{1}{(u_1+p)\cdots (u_{\nu-j}+p)} \nonumber \\
\equiv&\  \frac{1}{p} H(\{1\}^{j-1})\bigg( H^{(p)}_{2p-1}(\{1\}^{n-j-1})
    -p\sum_{i=0}^{n-j-2} H^{(p)}_{2p-1}(\{1\}^i,2,\{1\}^{n-i-j-2}) \nonumber \\
&\  \hskip2cm
-\sum_{i=1}^{\nu-j-2} \sum_{0<u_1<\cdots<u_{\nu-j-1}<p}  \frac{1}{(u_1+p)\cdots(u_i+p)(u_i+2p)\cdots (u_{\nu-j-1}+2p)}  \bigg) \pmod{p^2} \nonumber \\
\equiv&\ \frac{1}{p} H(\{1\}^{j-1})\bigg(
\frac{U_2(\{1\}^{n-j-1})}{(n-j-1)!}
    -\frac{U_2(2,\{1\}^{n-j-2})}{(n-j-2)!}p-\frac{U_2(2,\{1\}^{n-j-2})}{(n-j-2)!}\nonumber  \\
&\  +p\sum_{i=1}^{\nu-j-2} H(\{1\}^{i-1},3,\{1\}^{n-i-j-2})
   +p\sum_{i=1}^{\nu-j-2}\sum_{k=0}^{i-1} H(\{1\}^k,2,\{1\}^{i-k-2},2,\{1\}^{n-i-j-2})  \nonumber \\
&\  +2p\sum_{i=1}^{\nu-j-2} H(\{1\}^{i-1},3,\{1\}^{n-i-j-2})
   +2p\sum_{i=1}^{\nu-j-2}\sum_{k=0}^{n-i-j-2} H(\{1\}^{i-1},2,\{1\}^k,2,\{1\}^{n-i-j-k-3})
 \bigg)  \nonumber \\
\equiv&\  0  \pmod{p^2}
\end{align}
by Lemma~\ref{lem:Ub} and Corollary~\ref{cor:homo} since one of $j-1$ and $n-j-1$
is even and the other is odd.

 Similarly, for $2\le j\le n-2$,
\begin{align}\label{sum2}
\ &\sum_{\substack{1\le u_1<\dots<u_\nu<3p, \: u_j=2p
        \\ \forall k<j, u_k,u_2-u_1,\dots, u_{\nu}-u_{n-2} \in \calP_p }}  \frac{1}{u_1\dots u_\nu}  \nonumber \\
=&\ \frac{1}{2p} \sum_{\substack{1 \le u_1<\cdots<u_{j-1}<2p\\
    \forall k<j, u_k,u_2-u_1,\dots, u_{j-1}-u_{j-2} \in \calP_p }} \frac{1}{u_1\cdots u_{j-1}}
\sum_{2p<u_{j+1}<\cdots<u_\nu<3p} \frac{1}{u_{j+1}\cdots u_\nu} \nonumber \\
=&\ \frac{1}{2p}\left( H^{(p)}_{2p-1}(\{1\}^{j-1})-
\sum_{i=1}^{j-2} \sum_{1 \le u_1<\cdots<u_{j-2}<p} \frac{1}{u_1\cdots u_i(u_i+p)\cdots (u_{j-2}+p)} \right)
\nonumber \\
&\  \hskip3cm \times\sum_{0<u_1<\cdots<u_{n-j-1}<p} \frac{1}{(u_1+2p)\cdots (u_{n-j-1}+2p)} \nonumber  \\
\equiv&\  \frac{1}{2p} \bigg(H^{(p)}_{2p-1}(\{1\}^{j-1})
-\sum_{i=1}^{j-2} H(\{1\}^{i-1},2,\{1\}^{j-i-2})
+p\sum_{i=1}^{j-2}\sum_{k=0}^{j-i-2}  H(\{1\}^{i-1},2,\{1\}^k,2,\{1\}^{j-i-k-3}) \bigg) \nonumber \\
&\  \hskip3cm \times\bigg(H(\{1\}^{n-j-1})
    -2p\sum_{i=0}^{n-j-2} H(\{1\}^i,2,\{1\}^{n-i-j-2})\bigg) \pmod{p^2} \nonumber \\
\equiv&\  \frac{1}{2p} \left(\frac{U_2(\{1\}^{j-1})}{(j-1)!}
-\frac{U_1(\{1\}^{j-3})}{(j-3)!}+\frac{U_1(2,2,\{1\}^{j-4})}{2!(j-4)!}p \right)
\left(H(\{1\}^{n-j-1})-\frac{2U_1(2,\{1\}^{n-j-2})}{(n-j-2)!}p\right) \nonumber  \\
\equiv&\  0  \pmod{p^2}.
\end{align}
Further, for all $2\le i\le j-2\le n-4$, we obtain
\begin{align*}
&\ \sum_{\substack{1\le u_1<\dots<u_\nu<3p, \: u_i=p,u_j=2p
 \\ \forall k\ne i,k\ne j, u_k,u_2-u_1,\dots, u_{\nu}-u_{n-2} \in \calP_p }} \frac{1}{u_1\dots u_\nu}\\
=&\ \frac{1}{2p^2} \sum_{1\le u_1<\cdots<u_{i-1}<p} \frac{1}{u_1\cdots<u_{i-1}}
\sum_{p<u_{i+1}<\cdots<u_{j-1}<2p} \frac{1}{u_{i+1}\cdots u_{j-1}}
\sum_{2p<u_{j+1}<\cdots<u_\nu<3p} \frac{1}{u_{j+1}\cdots u_\nu} \\
\equiv&\ \frac{1}{2p^2} H(\{1\}^{i-1})
\Big(H(\{1\}^{j-i-1})-p\sum_{\ell=1}^{j-i-1} H(\{1\}^\ell,2,\{1\}^{j-i-\ell-2})\Big) \\
&\  \hskip3cm \times\Big(H(\{1\}^{n-j-1})-2p\sum_{\ell=1}^{n-j-1} H(\{1\}^\ell,2,\{1\}^{n-j-\ell-2})\Big) \\
\equiv&\ \frac{1}{2p^2} H(\{1\}^{i-1})H(\{1\}^{j-i-1})H(\{1\}^{n-j-1}) \\
-&\ \frac{1}{2p}H(\{1\}^{i-1})H(\{1\}^{n-j-1})\frac{U_1(2,\{1\}^{j-i-2})}{(j-i-2)!}
-\frac{1}{p}H(\{1\}^{i-1})H(\{1\}^{j-i-1})\frac{U_1(2,\{1\}^{n-j-2})}{(n-j-2)!}\\
\equiv&\  \frac{1}{2p^2} H(\{1\}^{i-1})H(\{1\}^{j-i-1})H(\{1\}^{n-j-1})  \pmod{p^2}.
\end{align*}
Thus by Corollary \ref{cor:homo} we deduce that
\begin{eqnarray}\label{addcong}
&&\sum_{i=2}^{n-4}\sum_{j=i+2}^{n-2} \sum_{\substack{1\le u_1<\dots<u_\nu<3p, \: u_i=p,u_j=2p
 \\ \forall k\ne i,k\ne j, u_k,u_2-u_1,\dots, u_{\nu}-u_{n-2} \in \calP_p }} \frac{1}{u_1\dots u_\nu}  \nonumber \\
&\equiv & \frac{1}{2p^2}\sum_{i=2}^{n-4}\sum_{j=i+2}^{n-2} H(\{1\}^{i-1})H(\{1\}^{j-i-1})H(\{1\}^{n-j-1})  \nonumber \\
&\equiv & \frac{1}{2p^2}\sum\limits_{\substack{a+b+c=n-3 \\a,b,c \ge 1}}H(\{1\}^{a})H(\{1\}^{b})H(\{1\}^{c}) \nonumber \\
&\equiv & \frac{1}{2p^2}\sum\limits_{\substack{a+b+c=\frac{n-3}{2} \\a,b,c \ge 1}} H(\{1\}^{2a})H(\{1\}^{2b})H(\{1\}^{2c}) \nonumber \\
&\equiv & -\frac{p}{2}\sum\limits_{\substack{a+b+c=\frac{n-3}{2} \\a,b,c \ge 1}} \frac{B_{p-2a-1}B_{p-2b-1}B_{p-2c-1}}{(2a+1)(2b+1)(2c+1)} \pmod{p^2}.
\end{eqnarray}

For the first sum in \eqref{equ:3pcaseSomeArePMultiple}, by the inclusion-exclusion principle,
\begin{eqnarray}\label{equ:3pcase5Ts}
&&\sum_{\substack{1\le u_1<\dots<u_\nu<3p \\ u_1,\dots,u_\nu\in \calP_p \\ u_2-u_1,\dots,u_\nu-u_{n-2} \in \calP_p}} \frac{1}{u_1\dots u_\nu} \nonumber\\
&\equiv &  \sum_{\substack{1\le u_1<\dots<u_\nu<3p \\ u_1,\dots,u_\nu\in \calP_p }} \frac{1}{u_1\dots u_\nu}
-\sum_{j=1}^{n-2} D_j-\sum_{j=1}^{n-2} T_j+\sum_{j=1}^{n-3}\sum_{k=j+2}^{n-2} T_{j,k}+\sum_{j=1}^{n-3} W_j  \nonumber \\
&\equiv & \frac{1}{(n-1)!} U_3(\{1\}^{n-1})
-\sum_{j=1}^{n-2} D_j-\sum_{j=1}^{n-2} T_j+\sum_{j=1}^{n-3}\sum_{k=j+2}^{n-2} T_{j,k}+\sum_{j=1}^{n-3} W_j  \pmod{p^2},
\end{eqnarray}
where (setting $v_{n-1}=3p$)
\begin{align*}
D_j=&\ \sum_{\substack{1 \le v_1<\dots<v_j<v_j+2p<v_{j+1}<\dots<3p \\ v_1,\dots,v_{n-2} \in \calP_p }} \frac{1}{v_1 \dots v_j(v_j+2p) v_{j+1}\dots v_{n-2}}, \\
T_j=&\ \sum_{\substack{1 \le v_1<\dots<v_j<v_j+p<v_{j+1}<\dots<3p \\ v_1,\dots,v_{n-2} \in \calP_p }} \frac{1}{v_1 \dots v_j(v_j+p) v_{j+1}\dots v_{n-2}}, \\
T_{j,k}=&\ \sum_{\substack{1 \le v_1<\cdots<v_j<v_j+p<v_{j+1}<\cdots<v_k<v_k+p<\cdots<3p \\ v_1,\cdots,v_{n-3} \in \calP_p }} \frac{1}{v_1 \cdots v_j(v_j+p) v_{j+1}\cdots v_k(v_k+p) v_{k+1} \cdots v_{n-3}}, \\
W_j=&\ \sum_{\substack{1 \le v_1<\cdots<v_j<v_j+p<v_j+2p<v_{j+1}<\cdots<3p \\ v_1,\cdots,v_{n-3} \in \calP_p }} \frac{1}{v_1 \cdots v_j(v_j+p) (v_j+2p) v_{j+1} \cdots v_{n-3}}.
\end{align*}
We have
\begin{align*}
D_j=&\ \sum_{\substack{1 \le v_1<\dots<v_{n-2}<p \\ v_1,\dots,v_{n-2} \in \calP_p }} \frac{1}{v_1 \dots v_j(v_j+2p) (v_{j+1}+2p)\dots (v_{n-2}+2p)}  \\
\equiv& \
H(\{1\}^{j-1},2,\{1\}^{n-j-2})-2p\bigg(H(\{1\}^{j-1},3,\{1\}^{n-j-2}) \\
+& \ \sum_{i=0}^{n-j-3} H(\{1\}^{j-1},2,\{1\}^i,2,\{1\}^{n-j-i-3})\bigg) \pmod{p^2}.
\end{align*}
So by Lemma~\ref{lem:Ub} we have
\begin{align*}
\sum_{j=1}^{n-2} D_j=&\ \frac{U_1(2,\{1\}^{n-3})}{(n-3)!}
-\frac{2U_1(3,\{1\}^{n-3})}{(n-3)!}p-\frac{2U_1(2,2,\{1\}^{n-4})}{(n-4)!}p
\equiv \frac{n-1}{n} B_{p-n} \cdot p  \pmod{p^2}.
\end{align*}
Similarly,
\begin{align*}
T_j=&\ \sum_{\substack{1 \le v_1<\dots<v_{n-2}<2p \\ v_1,\dots,v_{n-2} \in \calP_p }} \frac{1}{v_1 \dots v_j(v_j+p) (v_{j+1}+p)\dots (v_{n-2}+p)}  \\
\equiv& \
H^{(p)}_{2p-1}(\{1\}^{j-1},2,\{1\}^{n-j-2})-p\bigg(H^{(p)}_{2p-1}(\{1\}^{j-1},3,\{1\}^{n-j-2}) \\
+& \ \sum_{i=0}^{n-j-3} H^{(p)}_{2p-1}(\{1\}^{j-1},2,\{1\}^i,2,\{1\}^{n-j-i-3})\bigg) \pmod{p^2}.
\end{align*}
So by Lemma~\ref{lem:Ub} we have
\begin{align*}
\sum_{j=1}^{n-2} T_j=&\ \frac{U_2(2,\{1\}^{n-3})}{(n-3)!}
-\frac{U_2(3,\{1\}^{n-3})}{(n-3)!}p-\frac{U_2(2,2,\{1\}^{n-4})}{(n-4)!}p
\equiv \frac{2(n-1)}{n} B_{p-n} \cdot p  \pmod{p^2}.
\end{align*}
Moreover,
\begin{align*}
T_{j,k}=&\ \sum_{\substack{1 \le v_1<\cdots<v_{n-3}<p \\ v_1,\cdots,v_{n-3} \in \calP_p }} \frac{1}{v_1 \cdots v_j(v_j+p) (v_{j+1}+p)\cdots(v_k+p)(v_k+2p)\cdots(v_{n-3}+2p)}  \\
\equiv& \
H(\{1\}^{j-1},2,\{1\}^{k-j-1},2,\{1\}^{n-k-3})
-p\bigg(H(\{1\}^{j-1},3,\{1\}^{k-j-1},2,\{1\}^{n-k-3}) \\
+& \ \sum_{i=0}^{k-j-2} H(\{1\}^{j-1},2,\{1\}^i,2,\{1\}^{k-i-j-2},2,\{1\}^{n-k-3})
+H(\{1\}^{j-1},2,\{1\}^{k-j-1},3,\{1\}^{n-k-3})\bigg) \\
-& \ 2p\bigg(H(\{1\}^{j-1},2,\{1\}^{k-j-1},3,\{1\}^{n-k-3}) \\
+& \ \sum_{i=0}^{n-k-4} H(\{1\}^{j-1},2,\{1\}^{k-j-1},2,\{1\}^i,2,\{1\}^{n-i-k-4})\bigg)
\pmod{p^2}.
\end{align*}
Finally
\begin{align*}
W_j=&\ \sum_{\substack{1 \le v_1<\cdots<v_{n-3}<p \\ v_1,\cdots,v_{n-3} \in \calP_p }}
    \frac{1}{v_1 \cdots v_j(v_j+p)(v_j+2p)(v_{j+1}+2p)\cdots(v_{n-3}+2p)}  \\
\equiv& \
H(\{1\}^{j-1},3,\{1\}^{n-j-3})-pH(\{1\}^{j-1},4,\{1\}^{n-j-3}) \\
-& \ 2p\bigg(H(\{1\}^{j-1},4,\{1\}^{n-j-3})
+\sum_{i=0}^{n-k-4} H(\{1\}^{j-1},3,\{1\}^i,2,\{1\}^{n-i-j-4})\bigg)
\pmod{p^2}.
\end{align*}
Thus
\begin{multline*}
    \sum_{j=1}^{n-3}\sum_{k=j+1}^{n-2} T_{j,k}+\sum_{j=1}^{n-3} W_j \equiv
  \frac{U_1(2,2,\{1\}^{n-5})}{2(n-5)!}+\frac{U_1(3,\{1\}^{n-4})}{(n-4)!} \\
  -3p\left(\frac{U_1(4,\{1\}^{n-4})}{(n-4)!}+\frac{U_1(3,2,\{1\}^{n-5})}{(n-5)!}
  +\frac{U_1(\{2\}^3,\{1\}^{n-6})}{3!(n-6)!}\right) \pmod{p^2}\\
  \equiv -(n-4)!\frac{(n-1)}n \left(\frac{1}{2(n-5)!}+\frac{1}{(n-4)!}\right) B_{p-n} p
  \equiv - \frac{(n-1)(n-2)}{2n}  B_{p-n} p \pmod{p^2}.
\end{multline*}
Plugging this into \eqref{equ:3pcase5Ts}, we have
\begin{eqnarray}\label{firtsumfinalresult}
\sum_{\substack{1\le u_1<\dots<u_\nu<3p \\ u_1,\dots,u_\nu\in \calP_p \\ u_2-u_1,\dots,u_\nu-u_{n-2} \in \calP_p}} \frac{1}{u_1\dots u_\nu}
&\equiv&  -\frac{n!}{3}  \frac{(n^2+3n+2)}{2n} B_{p-n} \nonumber \\
&\equiv&  -\frac{(n+1)(n+2)}{6} (n-1)! B_{p-n} \pmod{p}
\end{eqnarray}
by Lemma~\ref{lem:Ub} again.

Now plugging (\ref{sum1}), (\ref{sum2}), (\ref{addcong}) and (\ref{firtsumfinalresult}) into (\ref{equ:3pcaseSomeArePMultiple}), and then combining with (\ref{starteq}), we get the desired result.
\end{proof}

\begin{cor} \label{cor:S_n(3)(p)forOddn}
Let $n\ge 3$ be an odd positive integer.
Then for all prime $p\ge \max\{n,5\}$, we have
\begin{equation*}
S_n^{(3)}(p)\equiv -\frac1n \binom{n}{3}\cdot (n-1)!B_{p-n}-\frac{n!}{6}\sum\limits_{\substack{a+b+c=\frac{n-3}{2} \\a,b,c \ge 1}} \frac{B_{p-2a-1}B_{p-2b-1}B_{p-2c-1}}{(2a+1)(2b+1)(2c+1)}  \pmod{p}.
\end{equation*}
\end{cor}

\begin{proof}
We observe that
\begin{align*}
\sum_{\substack{l_1+\dots+l_n=3p\\ l_j\in \calP_p \, \forall j }} \frac{1}{l_1\dots l_n}
\equiv &\ \sum_{\substack{l_1+\dots+l_n=3p\\ l_j<p,\, l_j \in \calP_p \, \forall j }} \frac{1}{l_1\dots l_n}
+\binom{n}{2} \sum_{\substack{l_1+\dots+l_n=p\\ l_1,\dots,l_n <p }} \frac{1}{(l_1+p)(l_2+p)l_3 \dots l_n}  \\
+&\ n \sum_{\substack{l_1+\dots+l_n=p\\ l_1,\dots,l_n <p }} \frac{1}{(l_1+2p) l_2 \dots l_n}
+n \sum_{\substack{l_1+\dots+l_n=2p\\ l_1,\dots,l_n <p }} \frac{1}{(l_1+p) l_2 \dots l_n} \pmod{p}.
\end{align*}
So we deduce
\begin{alignat*}{3}
S_n^{(3)}(p) &\, \equiv \sum_{\substack{l_1+\dots+l_n=3p\\  l_j \in \calP_p \, \forall j}} \frac{1}{l_1 \dots l_n}
- \binom{n+1}{2}S_n^{(1)}(p)- n S_n^{(2)}(p)    \\
&\, \equiv -\frac1n \binom{n}{3}\cdot (n-1)!B_{p-n}-\frac{n!}{6}\sum\limits_{\substack{a+b+c=\frac{n-3}{2} \\a,b,c \ge 1}} \frac{B_{p-2a-1}B_{p-2b-1}B_{p-2c-1}}{(2a+1)(2b+1)(2c+1)}     \pmod{p}
\end{alignat*}
by Lemma~\ref{lem:3pGeneral}, since $S_n^{(1)}(p)\equiv -(n-1)! B_{p-n}\pmod{p}$
by Lemma~\ref{lem:homo1} and $S_n^{(2)}(p)\equiv -\frac{n-1}2 (n-1)! B_{p-n}\pmod{p}$ by
Corollary~\ref{cor:S_n(2)(p)forOddn}.
\end{proof}

%%%%%%%%%%%%%%%

\section{Proof of the main theorem}

First, we prove a special case of Theorem~\ref{thm:main}.

\begin{prop}\label{prop:specialCase}
For all $r\ge 1$ and prime $p>7$ we have
\begin{equation*}
    S_7^{(1)}(p^{r+1})\equiv -\frac{7!}{10} B_{p-7} p^r \pmod{p^{r+1}}.
\end{equation*}
\end{prop}
\begin{proof}
By Lemma~\ref{lem:recurrence}, for all $r\ge 1$, we have
\begin{align*}
S_n^{(m)}(p^{r+1})\equiv \sum_{a=1}^{n-1}
   \big( (-1)^{m-1}{n-2 \choose m-1}\gam_n(a) p+O(p^2) \big) S_n^{(a)}(p^r) \pmod{p^{r+1}}.
\end{align*}
Here the $O(p^2)$ means a quantity which remains a $p$-adic integer
after dividing by the $p^2$.
By induction on $r$ it is not hard to see that for all $m=1,\dots,n-1$, we have
\begin{equation*}
S_n^{(m)}(p^{r+1}) \equiv 0 \pmod{p^r}, \quad \text{ for all }r\ge 1.
\end{equation*}
Thus for all $m=1,\dots,n-1$, by Lemmas~\ref{lem:Cn(m)a} and \ref{lem:recurrence}, we have
\begin{alignat*}{3}
S_n^{(m)}(p^{r+1})  &\, \equiv \sum_{a=1}^{n-1} (-1)^{m-1}
            \binom{n-2}{m-1} \gam_n(a) p  S_n^{(a)}(p^r) & \pmod{p^{r+1}}\, \\
    &\, \equiv   (-1)^{m-1}\binom{n-2}{m-1} S_n^{(1)}(p^{r+1})   & \pmod{p^{r+1}}.
\end{alignat*}
Thus by Lemmas~\ref{lem:Cn(m)a} and \ref{lem:recurrence}, for all $r\ge 2$
\begin{alignat*}{3}
S_n^{(1)}(p^{r+1}) \equiv&\ \sum_{m=1}^{n-1} C_{a,p}^{m}(n) S_n^{(m)}(p^{r})   & \pmod{p^{r+1}} \, \\
    \equiv &\ \sum_{m=1}^{n-1}
     (-1)^{m-1}\binom{n-2}{m-1}p\gam_{n}(m)  S_n^{(1)}(p^r)   & \pmod{p^{r+1}}\, \\
    \equiv &\ \sum_{m=1}^{n-1}
    \frac{(n-m-1)!(m-1)! p}{(n-1)!}\binom{n-2}{m-1}  S_n^{(1)}(p^r)   & \pmod{p^{r+1}}\,  \\
   \equiv &\  p   S_n^{(1)}(p^r)                         & \pmod{p^{r+1}},
\end{alignat*}
which proves \eqref{equ:kInductionStep}.
Finally, by applying Lemma~\ref{lem:recurrence} when $n=7$, we get
\begin{alignat*}{3}
S_7^{(1)}(p^2) \equiv &\, \frac{p}{3} S_7^{(1)}(p) - \frac{p}{15}S_7^{(2)}(p)
  +\frac{p}{30} S_7^{(3)}(p^r) & \pmod{p^2}\\
 \equiv &\,  \left(-\frac{p}{3}-\frac{3p}{15}-\frac{5p}{30}\right) 6! B_{p-7}
 \equiv -\frac{7!}{10}  B_{p-7} p  & \pmod{p^2}
\end{alignat*}
by Lemma~\ref{lem:homo1}, Corollary~\ref{cor:S_n(2)(p)forOddn} and
Corollary~\ref{cor:S_n(3)(p)forOddn}.
\end{proof}

We are now ready to prove Theorem~\ref{thm:main}.

Let $n=mp^{r}$, where $p$ does not divide $m$.
For any 7-tuples $(l_1,\cdots ,l_7)$ of integers satisfying $l_1+\cdots +l_7=n$, $l_{i} \in \calP_{p}$, $1 \le i \le 7$, we rewrite them as
\begin{equation*}
l_i=x_ip^r+y_i, \quad x_i \ge 0, \quad 1\le y_i<p^r, \quad y_i\in \calP_p, \quad 1 \le i \le 7.
\end{equation*}
Since
\begin{equation*}
\Big(\sum_{i=1}^7 x_i \Big)p^r+\sum_{i=1}^{7}{y_i}=mp^r,
\end{equation*}
we know there exists $1\le a \le 6$ such that
\begin{equation*}
\left\{ \begin{array}{ll}
 x_1+\cdots +x_7=m-a, \\
 y_1+\cdots +y_7=ap^{r}. \\
\end{array} \right.
\end{equation*}
For $1 \le a \le 6$, the equation $x_1+\cdots + x_7=m-a$ has $\binom {m+6-a}{6}$
nonnegative integer solutions. Hence
\begin{align}\label{thm2rec}
\sum_{\substack{l_1+\cdots +l_7=mp^r \\
 l_1,\cdots ,l_7\in \calP_p}}  \frac{1}{l_1l_2\cdots l_7}
& =\sum_{a=1}^6 \ \sum_{
 x_1+\cdots +x_7=m-a}   \   \sum_{\substack{
 y_1+\cdots +y_7=ap^r \\
 y_i\in \calP_p,y_i<p^r}}  \frac{1}{(x_1p^r+y_1)\cdots (x_7p^r+y_7)} \nonumber \\
&\equiv \sum_{a=1}^{6} \binom{m+6-a}{6}S_7^{(a)}(p^r)  \pmod{p^{r}}.
\end{align}

(i) If $r=1$, then since $S_7^{(1)}(p)\equiv -6!B_{p-7} \pmod{p}$. We also have $S_7^{(2)}(p) \equiv 3\cdot 6!B_{p-7} \pmod{p}$, $S_7^{(3)}(p)\equiv -5\cdot 6!B_{p-5}\pmod{p}$ and $S_7^{(a)}(p) \equiv -S_7^{(7-a)}(p) \pmod{p}$ for $4 \le a \le 6$. Hence from \eqref{thm2rec} we have
\begin{equation*}
\sum_{\substack{l_1+\cdots +l_7=n \\
 l_1,\cdots ,l_7\in \calP_p}}  \frac{1}{l_1l_2\cdots l_7}
 \equiv \frac{1}{6!}\Big(504m+210m^3+6m^5 \Big) S_7^{(1)}(p) \pmod{p}.
\end{equation*}
Since $S_7^{(1)}(p)\equiv -6!B_{p-7} \pmod{p}$ we complete the proof of (i).

(ii) If $r \ge 2$, then we have $S_7^{(2)}(p^r)\equiv -5S_7^{(1)}(p^{r}) \pmod{p^r}$ and $S_7^{(3)}(p^r)\equiv 10S_7^{(1)}(p^{r}) \pmod{p^r}$. Meanwhile, we have $S_7^{(a)}(p) \equiv -S_7^{(7-a)}(p)\pmod{p^r}$ for $4 \le a \le 6$.
Hence from (\ref{thm2rec}) we obtain
\begin{equation*}
\sum_{\substack{l_1+\cdots +l_7=n \\
 l_1,\cdots ,l_7\in \calP_p}}  \frac{1}{l_1l_2\cdots l_7}
 \equiv  \sum_{a=0}^5 (-1)^a \binom{5}{a}\binom{m+5-a}{6} S_7^{(1)}(p^{r})
 \equiv m S_7^{(1)}(p^{r}) \pmod{p^{r}}.
\end{equation*}
Since $S_7^{(1)}(p^{r}) \equiv -\frac{7!}{10}p^{r-1}B_{p-7} \pmod{p^r}$
by Proposition~\ref{prop:specialCase}, we complete the proof of (ii).

\section{Concluding remarks}
Using similar ideas from \cite{Zhao2014} we find that
it is unlikely to further generalize our main result to congruence \eqref{equ:CongruenceDepthd}
for $r\ge 2$, odd integer $d\ge 9$, and $q_d\in\Q$ depending only on $d$.
By using PSLQ algorithm we find that both the numerator and the denominator
of $q_9$ would have at least 60 digits if
the congruence \eqref{equ:CongruenceDepthd} holds for every prime $p\ge 11$. However, when $r=1$ we have obtained
a few general congruences in
Lemma \ref{lem:homo1}, Lemma~\ref{lem:2pGeneral} and Corollary~\ref{cor:S_n(2)(p)forOddn}, which can be
rephrased as follows.
Let $m=1,2$ and $d$ be any odd integer greater than $2$.
Then for any prime $p>d$, we have
\begin{equation}\label{equ:CongruenceDepth1S}
S_d^{(m)}(p) \equiv c_{d,m} \cdot (d-1)! B_{p-d} \pmod{p},
\end{equation}
where $c_{d,1}=-1$ and $c_{d,2}=(d-1)/2$, and
\begin{equation}\label{equ:CongruenceDepth1}
R_d^{(m)}(p) \equiv c'_{d,m} \cdot (d-1)! B_{p-d} \pmod{p},
\end{equation}
where $c'_{d,1}=-1$ and $c'_{d,2}=-(d+1)/2$.
Unfortunately, Lemma~\ref{lem:3pGeneral} and Corollary~\ref{cor:S_n(3)(p)forOddn} imply that these do not generalize to $m\ge 3$.
Computation with PSLQ algorithm suggests that if
\eqref{equ:CongruenceDepth1S} and \eqref{equ:CongruenceDepth1}
hold for $d=9,11,13,15$, $m=3,4$ then both the numerators and the denominators
of $c_{d,m}$ and $c'_{d,m}$ would have at least 60 digits. In fact,
numerical evidence suggests the following conjecture.
\begin{conj}\label{conj:wt8910}
For any prime $p\ge 11$, we have
\begin{align*}
R_8^{(m)}(p)\equiv&\, \frac{112}{5}m(m^2+16)(m^2-1) B_{p-3}B_{p-5}  \pmod{p},\\
R_9^{(m)}(p) \equiv&\, -\frac{8!}{18}\binom{m+2}{5} B_{p-3}^3
-8m(m^6+126 m^4+1869 m^2+3044) B_{p-9} \pmod{p},\\
R_{10}^{(m)}(p) \equiv&\, -\frac{24}{35}m(m^4+71m^2+540)(m^2-1)
\left(50 B_{p-3} B_{p-7}+21 B_{p-5}^2\right)\pmod{p}.
\end{align*}
\end{conj}
This conjecture is consistent with the general philosophy we have observed
for the finite multiple zeta values (FMZVs).
See, for example, \cite{Zhao2011c,Zhao2015a} for the definition of FMZVs and
the relevant results.
Note that according to the dimension conjecture of FMZVs
discovered by Zagier and independently by the last author
(see \cite{Zhao2015a}) the weight 8 (resp. weight 10)
piece of FMZVs has conjectural dimension 2 (resp. 3).
Theorem~\ref{thm:main} (i), Conjecture~\ref{conj:wt8910}
and all the previous works in lower weights
imply that $R_d^{(m)}(p)$ ($d\le 10$ and $m\ge 2$) should lie in the proper
subalgebra generated by the so-called ${\mathcal A}_1$-Bernoulli numbers
defined in \cite{Zhao2015a}. According to the analogy between FMZVs
and MZVs, this subalgebra is the FMZV analog
of the MZV subalgebra generated by the Riemann zeta values. It would be
interesting to see if this phenomenon holds in every weight.

\bigskip
\noindent
{\bf Acknowledgements.} JZ is partially supported by the NSF grant DMS~1162116.
Part of this work was done while he was visiting the Max Planck Institute for
Mathematics, IHES and ICMAT at Madrid, Spain, whose supports are gratefully acknowledged.
The authors also thank the anonymous referee for a number of valuable suggestions
which improved the paper greatly.

\end{document}